\documentclass[a4paper,12pt]{amsart}

\makeindex

\usepackage{amsfonts,graphics,amsmath,amsthm,amsfonts,amscd,amssymb,amsmath,latexsym,euscript, enumerate}
\usepackage{epsfig,url}
\usepackage{flafter}
\usepackage[all,cmtip,line]{xy}
\usepackage{array}
\usepackage[english]{babel}
\usepackage{overpic}
\usepackage{subfig}
\usepackage{multirow}

\usepackage{wrapfig}
\usepackage{enumitem}

\usepackage[usenames,dvipsnames,svgnames,table]{xcolor}
\usepackage{graphicx}
\usepackage[bookmarks=true,bookmarksnumbered=true,
 pdftitle={},
 pdfsubject={},
 pdfauthor={},
 pdfkeywords={TeX, dvipdfmx, hyperref, color,},
 colorlinks=true, linkcolor=RoyalBlue, citecolor=Emerald]
 {hyperref}

\newtheorem{theorem}{Theorem}[section]
\newtheorem{lemma}[theorem]{Lemma}
\newtheorem{proposition}[theorem]{Proposition}

\newtheorem{question}[theorem]{Question}

\theoremstyle{corollary}
\newtheorem{corollary}[theorem]{Corollary}

\theoremstyle{definition} % italic or bold etc.

\newtheorem{definition-lemma}[theorem]{Definition-Lemma}

\newtheorem{example}[theorem]{Example}

\theoremstyle{remark}
\newtheorem{remark}[theorem]{Remark}

\numberwithin{equation}{section}

\newcommand{\C}{\mathbb{C}}
\newcommand{\R}{\mathbb{R}}
\newcommand{\Z}{\mathbb{Z}}

\newcommand{\Q}{\mathbb{Q}}
\newcommand{\mc}{\mathcal}

\DeclareMathOperator{\ord}{ord}

\DeclareMathOperator{\vol}{vol}

\DeclareMathOperator{\SB}{SB}

\newcommand{\bm}{\mathbf B_-}  %B-minus

\newcommand{\bp}{\mathbf B_+}  %B-plus
\newcommand{\okbd}{\Delta}

\newcommand{\oklim}{\Delta^{\lim}}

\def\res{\operatorname{res}}
\def\zar{\operatorname{Zar}}

\def\BDPP{\operatorname{BDPP}}
\def\bir{\operatorname{bir}}

\def\Proj{\operatorname{Proj}}

\def\Supp{\operatorname{Supp}}

\newcommand{\vleq}{\rotatebox[origin=c]{90}{$\leq$}}

\title[Comparing numerical Iitaka dimensions again]
{Comparing numerical Iitaka dimensions again}

\begin{document}

\author{Sung Rak Choi}
\address{Department of Mathematics, Yonsei University, 50 Yonsei-ro, Seodaemun-gu, Seoul 03722, Republic of Korea}
\email{sungrakc@yonsei.ac.kr}

\author{Jinhyung Park}
\address{Department of Mathematics, Sogang University, 35 Baekbeom-ro, Mapo-gu, Seoul 04107, Republic of Korea}
\email{parkjh13@sogang.ac.kr}

%\thanks{The first author was partially supported by}
%\subjclass[2010]{}
\date{\today}

\begin{abstract}
To seek for the useful numerical analogues to the Iitaka dimension, various numerical Iitaka dimensions have been defined from a number of different perspectives. It has been accepted that all the known numerical Iitaka dimensions coincide with each other until the recent discovery of a counterexample constructed by Lesieutre. In this paper, we prove that many of them still coincide with the numerical Iitaka dimension introduced by Boucksom--Demailly--P\u{a}un--Peternell. On the other hand, we show that some other numerical Iitaka dimensions introduced by Nakayama and Lehmann can be arbitrarily larger than the rest of numerical Iitaka dimensions. We also study some properties of abundant divisors.
\end{abstract}

\maketitle
%\tableofcontents

%%%%%%%%%%%%%%%%%%%%%%%%%%%%%%%%%%%%%%%%%%%%%%%%%%%%%
\section{Introduction}
%%%%%%%%%%%%%%%%%%%%%%%%%%%%%%%%%%%%%%%%%%%%%%%%%%%%%

Throughout the paper, we work over the field $\C$ of complex numbers. Let $X$ be a $\Q$-factorial normal projective variety of dimension $n$, and $D$ be a divisor on $X$. By a divisor on $X$, we mean an $\R$-divisor on $X$. If $D$ is effective, then a nonnegative integer $\kappa(D)$, called the \emph{Iitaka dimension} of $D$, can be uniquely defined and it is known that there exist constants $C_1, C_2 > 0$ for which
\begin{equation}\label{eq:iitakadim}
C_1 m^{\kappa(D)} < h^0(X, \lfloor mD \rfloor) < C_2 m^{\kappa(D)}
\end{equation}
holds for every sufficiently large integer $m>0$. The Iitaka dimension is a \emph{birational invariant} in the sense that if $f \colon Y \to X$ is a birational morphism from a smooth projective variety $Y$, then $\kappa(f^*D)=\kappa(D)$. However, it is well known that $\kappa(D)$ is not a numerical invariant.
\medskip

There are several numerical analogues to the Iitaka dimension defined for a pseudoeffective divisor $D$.
Boucksom--Demailly--P\u{a}un--Peternell \cite{BDPP} defined the following numerical Iitaka dimension:
$$
\nu_{\BDPP}(D):= \max \{ k \in \Z_{\geq 0} \mid \langle D^k \rangle \neq 0 \}
$$
where $\langle - \rangle$ denotes the positive intersection product. If $D$ is nef, then $\langle D^k \rangle$ is the usual self-intersection number $[D^k]$ so that $\nu_{\BDPP}(D)$ is nothing but Kawamata's numerical dimension $\nu(D)$ in \cite{K}.
From now on, let $D$ be a pseudoeffective divisor, and $A$ be a fixed sufficiently positive ample $\Z$-divisor on $X$. Nakayama \cite{nakayama} proposed the following definitions of numerical Iitaka dimension:
$$
\arraycolsep=1.4pt\def\arraystretch{1.9}
\begin{array}{rl}
\kappa_{\sigma}(D)&:=\displaystyle  \max \left\{ k \in \Z_{\geq 0} \left| \limsup\limits_{m \to \infty} \frac{h^0(X, \lfloor mD \rfloor + A)}{m^k} > 0 \right.\right\} \\
\kappa_{\nu}(D)&:=\min \{ \dim W \mid \text{$D$ does not numerically dominate $W$}\}.
\end{array}
$$
Whenever we consider the situation that $m \to \infty$ as in the definition of $\kappa_{\sigma}$, we assume that $m$ is an integer.
Below, $\phi \colon (\widetilde{X}, \widetilde{W}) \to (X, W)$ range over all $W$-birational models, which by definition means that $W$ is not contained in any $\phi$-exceptional center and $\widetilde{W}$ is the strict transform of $W$ (see \cite[Definition 2.10]{lehmann-nu}). Using the volumes of divisors, Lehmann \cite{lehmann-nu} further introduced the following numerical Iitaka dimensions:
$$
\arraycolsep=1.4pt\def\arraystretch{1.9}
\begin{array}{rl}
\kappa_{\vol, \res}(D)&:=\max\left\{ \dim W \left| \vol^+_{X|W}(D) > 0\right. \right\}\\
\kappa_{\vol, \zar}(D)&:=\max \left\{ \dim W \left| \inf\limits_{\phi} \vol_{\widetilde{W}}(P_{\sigma}(\phi^* D)|_{\widetilde{W}}) > 0 \right.\right\}\\
\kappa_{\vol}(D)&:=\displaystyle \max \left\{k\in \Z_{\geq 0}\left| \liminf\limits_{\varepsilon \to 0} \frac{\vol_X(D+ \varepsilon A)}{\varepsilon^{n-k}} > 0\right. \right\}.
\end{array}
$$
We set $\kappa_{\sigma, \sup}(D):=\kappa_{\sigma}(D)$ and $\kappa_{\vol, \inf}(D):=\kappa_{\vol}(D)$, and we consider the following variants of $\kappa_{\sigma}(D)$ and $\kappa_{\vol}(D)$:
$$
\arraycolsep=1.4pt\def\arraystretch{1.9}
\begin{array}{rl}
\kappa_{\sigma, \inf}(D)&:=\displaystyle  \max \left\{ k \in \Z_{\geq 0} \left| \liminf\limits_{m \to \infty} \frac{h^0(X, \lfloor mD \rfloor + A)}{m^k} > 0 \right.\right\} \\
\kappa_{\vol, \sup}(D)&:=\displaystyle \max \left\{k\in \Z_{\geq 0}\left| \limsup\limits_{\varepsilon \to 0} \frac{\vol_X(D+ \varepsilon A)}{\varepsilon^{n-k}} > 0\right. \right\}.
\end{array}
$$
Lesieutre \cite{lesieutre} suggested to allow real values for the numerical Iitaka dimensions. Here we consider several $\R$-valued variants of $\kappa_{\sigma}(D)$ and $\kappa_{\vol}(D)$ as follows:
$$
\arraycolsep=1.4pt\def\arraystretch{1.9}
\begin{array}{rl}
\kappa_{\sigma, \inf}^{\R}(D)&:=\displaystyle \sup\left\{ k \in \R_{\geq 0} \left| \liminf\limits_{m \to \infty} \frac{h^0(X, \lfloor mD \rfloor + A)}{m^k} > 0 \right.\right\}\\
\kappa_{\sigma, \sup}^{\R}(D)&:=\displaystyle \sup \left\{ k \in \R_{\geq 0} \left| \limsup\limits_{m \to \infty} \frac{h^0(X, \lfloor mD \rfloor + A)}{m^k} > 0 \right.\right\}\\
\kappa_{\vol,\inf}^{\R}(D)&:=\displaystyle \sup \left\{k\in \R_{\geq 0}\left| \liminf\limits_{\varepsilon \to 0} \frac{\vol_X(D+ \varepsilon A)}{\varepsilon^{n-k}} > 0\right. \right\}\\
\kappa_{\vol,\sup}^{\R}(D)&:=\displaystyle \sup \left\{k\in \R_{\geq 0}\left| \limsup\limits_{\varepsilon \to 0} \frac{\vol_X(D+ \varepsilon A)}{\varepsilon^{n-k}} > 0\right. \right\}.
\end{array}
$$
Notice that $\kappa_{\star, \diamond}(D) \leq \kappa_{\star, \diamond}^{\R}(D)$, where $\star = \sigma$ or $\vol$ and $\diamond=\inf$ or $\sup$. If $\kappa_{\star, \diamond}^{\R}(D)$ is not an integer, then $\kappa_{\star, \diamond}(D)  = \lfloor \kappa_{\star, \diamond}^{\R}(D) \rfloor$. However, it is unclear whether $\kappa_{\star, \diamond}(D)  = \lfloor \kappa_{\star, \diamond}^{\R}(D) \rfloor$ always holds.
Below, $\phi \colon \widetilde{X} \to X$ range over all birational morphisms with $\widetilde{X}$ smooth projective. We further define variants of  $\kappa_{\sigma}(D)$ and $\kappa_{\vol}(D)$ as follows:
$$
\arraycolsep=1.4pt\def\arraystretch{1.9}
\begin{array}{rl}
\kappa_{\vol,\inf,\bir}^{\R}(D)&:=\displaystyle \sup \left\{k\in \R_{\geq 0}\left| \liminf\limits_{\varepsilon \to 0} \inf\limits_{\phi} \frac{\vol_{\widetilde{X}}(P_{\sigma}(\phi^*D)+ \varepsilon \phi^*A)}{\varepsilon^{n-k}} > 0\right. \right\}\\
\kappa_{\vol,\sup,\bir}^{\R}(D)&:=\displaystyle \sup \left\{k\in \R_{\geq 0}\left| \limsup\limits_{\varepsilon \to 0} \inf\limits_{\phi} \frac{\vol_{\widetilde{X}}(P_{\sigma}(\phi^*D)+ \varepsilon \phi^*A)}{\varepsilon^{n-k}} > 0\right. \right\}.
\end{array}
$$
Finally, we add one more definition of numerical Iitaka dimension. Choi--Hyun--Park--Won \cite{CHPW} introduced and studied the limiting Okounkov body $\oklim_{Y_\bullet}(D)$ of $D$ with respect to an admissible flag $Y_\bullet$ on $X$. Since the set of $\oklim_{Y_\bullet}(D)$ for all the admissible flags $Y_\bullet$ reflects the numerical properties of $D$, it is natural to consider the following numerical invariant:
$$
\kappa_{\Delta}(D) := \max\{ \dim \oklim_{Y_\bullet}(D) \mid Y_\bullet~\text{is an admissible flag on $X$} \}.
$$
Note that all of the definitions defined above depend only on the numerical class of $D$ and they are birational invariant. All the necessary notions required in the definitions of the numerical Iitaka dimensions are recalled in Section \ref{sec:prelim}.

\medskip

Note that when $D$ is nef, all of these notions coincide with Kawamata's numerical dimension $\nu(D)$.  From the main result of \cite{lehmann-nu}, we can extract the following equalities:
$$
\nu_{\BDPP}(D) = \kappa_{\vol, \res}(D)=\kappa_{\vol, \zar}(D).
$$
We first prove that some more numerical Iitaka dimensions also coincide with $\nu_{\BDPP}(D)$.

\begin{theorem}\label{thm:main}
Let $X$ be a $\Q$-factorial normal projective variety, and $D$ be a pseudoeffective $\R$-divisor on $X$. Then we have
$$
\nu_{\BDPP}(D)=\kappa_{\vol, \res}(D)=\kappa_{\vol,\zar}(D)= \kappa_{\vol, \inf, \bir}^{\R}(D) = \kappa_{\vol, \sup, \bir}^{\R}(D)  = \kappa_{\Delta}(D).
$$
\end{theorem}

For other numerical Iitaka dimensions, we only have the following inequalities:
$$
\begin{array}{ccccccc}
&&\kappa_{\vol, \inf}^{\R}(D) & \leq & \kappa_{\vol, \sup}^{\R}(D) & \leq & \kappa_{\nu}(D)\\
&& \vleq & & \vleq &&\\
\nu_{\BDPP}(D) & \leq & \kappa_{\sigma, \inf}^{\R}(D) & \leq & \kappa_{\sigma, \sup}^{\R}(D),&&
\end{array}
$$
and the same inequalities hold without $\R$-valued variants (see Proposition \ref{prop:ineqnumdim}). In particular, we have
$$
\begin{array}{ccc}
\kappa_{\vol}(D) & \leq & \kappa_{\nu}(D)\\
 \vleq & & \vleq \\
\nu_{\BDPP}(D) & \leq  & \kappa_{\sigma}(D).
\end{array}
$$

\medskip

Recently, Lesieutre \cite{lesieutre} showed that $\kappa_{\nu}(D)$ can be strictly larger than $\nu_{\BDPP}(D)$;  there exists a pseudoeffective divisor $D$ on a smooth projective threefold such that
$$
\nu_{\BDPP}(D)=1,~\kappa_{\sigma, \inf}^{\R}(D) = \kappa_{\sigma, \sup}^{\R}(D)=\kappa_{\vol, \inf}^{\R}(D)=\kappa_{\vol, \sup}^{\R}(D)=\frac{3}{2},~\kappa_{\nu}(D)=2.
$$
Note that in this example, $\nu_{\BDPP}(D)=\kappa_{\sigma}(D)=\kappa_{\vol}(D)=1$.
In this paper, we prove that $\kappa_{\sigma}(D)$ and $\kappa_{\vol}(D)$ can be arbitrarily larger than $\nu_{\BDPP}(D)$.
This gives an answer to Question 0.2 in the Errata of \cite{lehmann-nu}.

\begin{theorem}\label{thm:notcoincide}
For any given integer $k \geq 1$, there exist a smooth projective variety $X$ and a pseudoeffective $\R$-divisor $D$ on $X$ such that
$$
\kappa_{\sigma}(D) - \nu_{\BDPP}(D) \geq k~~ \text{and}~~ \kappa_{\vol}(D) - \nu_{\BDPP}(D) \geq k.
$$
\end{theorem}
In our example for Theorem \ref{thm:notcoincide} (Example \ref{ex:sigma,vol>bdpp}), $\kappa_{\vol}(D)=\kappa_{\sigma}(D)$ holds. Furthermore, we have
$$
\kappa_{\sigma, \inf}^{\R}(D) = \kappa_{\sigma, \sup}^{\R}(D)=\kappa_{\vol, \inf}^{\R}(D)=\kappa_{\vol, \sup}^{\R}(D).
$$
It would be interesting to know whether $\kappa_{\sigma, \inf}^{\R}(D) = \kappa_{\vol, \inf}^{\R}(D)$ and $\kappa_{\sigma, \sup}^{\R}(D) = \kappa_{\vol, \sup}^{\R}(D)$ hold in general (see Question \ref{ques:sigma=vol?} and Remark \ref{rem:newcounterex}).

\medskip

In view of the fundamental inequalities (\ref{eq:iitakadim}) for Iitaka dimensions, it is natural to ask whether there exist constants $C_1, C_2 > 0$ such that
$$
C_1 m^{\kappa_{\sigma, \inf}^{\R}(D)} < h^0(X, \lfloor mD \rfloor + A) < C_2 m^{\kappa_{\sigma, \sup}^{\R}(D)}
$$
for every sufficiently large integer $m > 0$ (see Question \ref{ques:fundineq} $(1)$).
Similarly, it is also natural to ask whether there exist constants $C_1, C_2 > 0$ such that
$$
C_1 \varepsilon^{n-\kappa_{\vol, \inf}^{\R}(D)} < \vol_X (D+\varepsilon A) < C_2 \varepsilon^{n-\kappa_{\vol, \sup}^{\R}(D)}
$$
for every sufficiently small number $\varepsilon > 0$ (see Question \ref{ques:fundineq} $(2)$). As an application of Theorem \ref{thm:main}, we show that the inequalities hold on the limit of higher birational models.

\begin{corollary}\label{cor:minmax}
Let $X$ be a $\Q$-factorial normal projective variety of dimension $n$, and $D$ be a pseudoeffective $\R$-divisor on $X$. Fix a sufficiently positive ample $\Z$-divisor $A$ on $X$.
Then there are constants $C_1, C_2 > 0$ such that
$$
C_1 \varepsilon^{n-\nu_{\BDPP}(D)} < \inf\limits_{\phi} \vol_{\widetilde{X}}(P_{\sigma}(\phi^*D)+\varepsilon \phi^*A) < C_2 \varepsilon^{n-\nu_{\BDPP}(D)}
$$
for every sufficiently small number $\varepsilon > 0$, $\phi \colon \widetilde{X} \to X$ in $\inf$ range over all birational morphisms with $\widetilde{X}$ smooth projective.
\end{corollary}

Notice that $\kappa(D) \leq \nu_{\BDPP}(D)$. We say that a divisor $D$ is \emph{abundant} if $\kappa(D)=\nu_{\BDPP}(D)$ holds (cf. \cite{BDPP}). This notion is weaker than the conditions $\kappa(D)=\kappa_{\sigma}(D)$ or $\kappa(D)=\kappa_{\nu}(D)$, which are also often referred to as the abundance in the literature (see e.g. \cite{DHP}, \cite{F}, \cite{GL}, \cite{nakayama}).
In this paper, we generalize the well-known result of Kawamata for nef and abundant divisors \cite[Proposition 2.1]{K} (see also the Errata of \cite{lehmann-red}).

\begin{theorem}\label{thm:abundantdivisor}
Let $X$ be a $\Q$-factorial normal projective variety, and $D$ be an effective $\R$-divisor on $X$ such that $\kappa(D) > 0$. Then $D$ is abundant in the sense that $\kappa(D) = \nu_{\BDPP}(D)$ holds if and only if there exist a birational morphism $\mu \colon W \to X$ from a smooth projective variety $W$ and a surjective morphism $g \colon W \to Z$ to a smooth projective variety $Z$ with connected fibers such that
$$
P_{\sigma}(\mu^*D) \sim_{\Q} P_{\sigma}(g^*B)
$$
for some big divisor $B$ on $Z$, where $g \colon W \to Z$ is a birational model of the Iitaka fibration of $D$.
\end{theorem}

One may wonder whether $\kappa(D) = \nu_{\BDPP}(D)$ implies $\kappa(D)=\kappa_{\nu}(D)$.
If $\kappa(D) = \nu_{\BDPP}(D)$, then Theorem \ref{thm:abundantdivisor} shows that $\kappa(D)=\kappa_{\nu}(P_{\sigma}(\mu^*D))$. At this moment, it is not clear to us  whether $\kappa_{\nu}(\mu^*D)=\kappa_{\nu}(P_{\sigma}(\mu^*D))$ always holds or not (see Question \ref{ques:numdimdivzar}).

\medskip

Let $(X, \Delta)$ be a $\Q$-factorial projective klt pair such that $K_X + \Delta$ is pseudoeffective.
The abundance conjecture, which predicts that $K_X + \Delta$ is abundant, is one of the most important open problems in birational geometry.
It is well known that $\kappa(K_X+\Delta)=\kappa_{\sigma}(K_X+\Delta)$ if and only if $(X,\Delta)$ has a good minimal model (see \cite[Remark 2.6]{DHP}, \cite[Theorem 4.3]{GL}; see also \cite[Section 3]{F}).
We refer to Section \ref{sec:abdiv} for the definitions of a klt pair and a good minimal model.
In this paper, we show that the weaker notion of the abundance for $(X, \Delta)$, that is $\kappa(K_X+\Delta)=\nu_{\BDPP}(K_X+\Delta)$, is enough to guarantee the existence of a good minimal model of $(X, \Delta)$.

\begin{theorem}\label{thm:abconj<=>goodmin}
Let $(X,\Delta)$ be a $\Q$-factorial projective klt pair such that $K_X+\Delta$ is a pseudoeffective $\Q$-divisor. Then $K_X+\Delta$ is abundant in the sense that $\kappa(K_X+\Delta)=\nu_{\BDPP}(K_X+\Delta)$ holds if and only if $(X,\Delta)$ has a good minimal model.
\end{theorem}

We remark that the $\Q$-factoriality on $X$ in the main theoerems was assumed only for simplicity.
Such condition can be removed in most of the results in this paper.

\medskip

The paper is organized as follows. We begin by recalling some basic notions for the definitions of the numerical Iitaka dimensions in Section \ref{sec:prelim}. In Section \ref{sec:numiitakadim}, we present several basic properties of numerical Iitaka dimensions, and we prove Theorem \ref{thm:main} and Corollary \ref{cor:minmax}. In Section \ref{sec:example}, we present an example of a pseudoeffective divisor $D$ on some smooth projective variety $X$ for which the inequalities of Theorem \ref{thm:notcoincide} hold. Section \ref{sec:abdiv} is devoted to proving Theorems \ref{thm:abundantdivisor} and \ref{thm:abconj<=>goodmin}.

%\subsection*{Acknowledgement}

%%%%%%%%%%%%%%%%%%%%%%%%%%%%%%%%%%%%%%%%%%%%%%%%%%%%%
\section{Preliminaries}\label{sec:prelim}
%%%%%%%%%%%%%%%%%%%%%%%%%%%%%%%%%%%%%%%%%%%%%%%%%%%%%

Most of the definitions of the numerical Iitaka dimensions in the introduction are self-explanatory.
Below, we recall the notions that are used in defining the numerical Iitaka dimensions. Throughout the section, $X$ is an $n$-dimensional $\Q$-factorial normal projective variety, $V \subseteq X$ is a $v$-dimensional irreducible closed subvariety, and $D$ is an $\R$-divisor on $X$.

\subsection{Divisorial Zariski decomposition}
We denote by $\ord_V(E)$ the order of an effective divisor $E$ along $V$.
If $D$ is big, then \emph{the asymptotic valuation} of $V$ at $D$ is defined as
$$
\ord_V(||D||):=\inf\{\ord_V(D')\mid D\equiv D'\geq 0\}.
$$
If $D$ is only pseudoeffective, then \emph{the asymptotic valuation} of $V$ at $D$ is defined as
$$
\ord_V(||D||):=\lim_{\varepsilon\to 0+}\ord_V(||D+\varepsilon A||)
$$
where $A$ is an ample divisor on $X$. One can check that this definition is independent of the choice of the ample divisor $A$. Note that $\ord_V(||D||)$ is a numerical invariant of $D$ and the function $\ord_V(|| \cdot ||) \colon \text{Big}(X) \to \R$ is continous on the cone of big divisor classes.
In particular, if $D$ is big, then $\ord_V(||D||)=\displaystyle\lim_{\varepsilon\to 0+}\ord_V(||D+\varepsilon A||)$ for any ample divisor $A$ on $X$.
Suppose that $D$ is pseudoeffective. By \cite[Corollary III.1.1]{nakayama}, there are only finitely many prime divisors $E$ on $X$ such that $\ord_E(||D||)>0$. The \emph{negative part} of $D$ is defined as $N_{\sigma}(D):=\sum_{E}  \ord_E(||D||)E$ where the summation is taken over all the finitely many prime divisors $E$ of $X$ such that $ \ord_E(||D||)>0$. Note that the negative part $N_{\sigma}(D)$ is a numerical invariant of $D$. The \emph{positive part} of $D$ is defined as $P_{\sigma}(D):=D-N_{\sigma}(D)$. By \cite[Proposition III.1.14]{nakayama}, the positive part $P_{\sigma}(D)$ can be characterized as the maximal movable divisor such that $P_{\sigma}(D) \leq D$. The \emph{divisorial Zariski decomposition} of $D$ is given by the decomposition
$$
D=P_{\sigma}(D) + N_{\sigma}(D).
$$
For more details, we refer to \cite{B}, \cite{elmnp1}, \cite[Chapter III]{nakayama}.

\subsection{Asymptotic base loci}
The \emph{stable base locus} of $D$ is defined as
$$
\SB(D):= \bigcap_{D \sim_{\R} D' \geq 0} \Supp(D').
$$
where $D\sim_\R D'$ denotes the $\R$-linear equivalence, that is, $D-D'$ is the $\R$-linear combination of principal divisors.
The \emph{augmented base locus} of $D$ is defined as
$$
\bp(D):=\bigcap_{A\text{:ample}}\text{SB}(D-A).
$$
The \emph{restricted base locus} of $D$ is defined as
$$
\bm(D):=\bigcup_{A\text{:ample}}\SB(D+A).
$$
It is well known that  $\bp(D)$ and $\bm(D)$ depend only on the numerical class of $D$.
Unlike $\mathbf B_+(D)$, Lesieutre  \cite{john} proved that $\mathbf B_-(D)$ is not always Zariski closed. In this case, the dimension $\dim\mathbf B_-(D)$ is defined as the maximum of the dimensions of the irreducible components of $\mathbf B_-(D)$.
%Note that $\bm(D)=X$  (resp. $\bp(D)=X$) if and only if $D$ is not pseudoeffective (resp. not big), and $\bm(D)=\emptyset$ (resp. $\bp(D)=\emptyset$) if and only if $D$ is nef (resp. ample).
The union of codimension one components of $\bm(D)$ coincides with  $\Supp(N_{\sigma}(D))$. For more details, we refer to \cite{elmnp1}.

\subsection{Restricted volumes}
The \emph{restricted volume} of $D$ along $V$ is defined as
$$
\vol_{X|V}(D):=\limsup_{m \to \infty} \frac{h^0(X|V,\lfloor mD \rfloor)}{m^v/v!}
$$
where $h^0(X|V, \lfloor mD \rfloor)$ is the dimension of the image of the natural restriction map
$$
H^0(X, \mc O_X(\lfloor mD \rfloor ))\to H^0(V,\mc O_V(\lfloor mD \rfloor|_V)).
$$
When $V=X$, we simply set
$$
\vol_X(D):=\vol_{X|X}(D) = \limsup_{m \to \infty} \frac{h^0(X, \lfloor mD \rfloor)}{m^n/n!}
$$
and we call it the \emph{volume} of $D$.
If $V \not\subseteq \bp(D)$, then `$\limsup$' can be replaced by `$\lim$' in the definition, the restricted volume $\vol_{X|V}(D)$ depends only on the numerical class of $D$, and it uniquely extends to a continuous function
$
\vol_{X|V} \colon \text{Big}^V (X) \to \R
$
where $\text{Big}^V(X)$ is the set of all $\R$-divisor classes $\xi$ such that $V$ is not properly contained in any irreducible component of $\bp(\xi)$ (see \cite[Corollaries 2.14, 2.15 and Theorem 5.2]{elmnp2}). Furthermore, in this case, $\vol_{X|V}(D)=\vol_{X|V}(P_{\sigma}(D))$ holds.
If $\phi \colon (\widetilde{X}, \widetilde{V}) \to (X, V)$ is a $V$-birational morphism, then $\vol_{\widetilde{X}|\widetilde{V}}(\phi^*D) = \vol_{X|V}(D)$ (see \cite[Lemma 2.4]{elmnp2}).
The \emph{augmented restricted volume} of $D$ along $V$ is defined as
$$
\vol_{X|V}^+(D):=\lim_{\varepsilon \to 0+} \vol_{X|V}(D+\varepsilon A),
$$
where $A$ is an ample divisor on $X$. It is easy to check that the definition is independent of the choice of the $A$ and $\vol_{X|V}^+(D)$ depends only on the numerical class of $D$. If $V \subseteq \bm(D)$, then $\vol_{X|V}^+(D)=0$ by definition. For more details, we refer to \cite{CHPW}, \cite{elmnp2}.

\subsection{Restricted positive intersection product}
For two cycle classes $K, K' \in N^k(Y)$ on a normal variety $Y$ of dimension $m$, we define $K \geq K'$ if $K-K'$ is contained in the closure of the cone generated by effective cycles of dimension $m-k$.
When $L_1, \ldots, L_k$ are big divisors on $X$ with $V \not\subseteq \bp(L_i)$ for all $1 \leq i \leq k$, we define the class
$$
\langle L_1 \cdots L_k \rangle_{X|V} \in N^k(V)
$$
as the maximum under the ordering $\leq$ of $\phi_* (N_1 \cdots N_K \cdot \widetilde{V})$, where $\phi \colon (\widetilde{X}, \widetilde{V}) \to (X, V)$ run over the smooth $V$-birational models,  $N_i$ are nef divisors on $\widetilde{X}$ such that $E_i:=\phi^*L_i - N_i \geq 0$ and $\widetilde{V} \not\subseteq \Supp(E_i)$ for each $1 \leq i \leq n$.
When $L_1, \ldots, L_k$ are pseudoeffective divisors on $X$ with $V \not\subseteq \bm(L_i)$ for all $1 \leq i \leq k$, we define the class
$$
\langle L_1 \cdots L_k \rangle_{X|V} = \lim_{m \to \infty} \langle (L_1+B_{1,m}) \cdots (L_k +  B_{k,m}) \rangle_{X|V} \in N^k(V)
$$
where $B_{i,m}$ are big divisors with $V \not\subseteq \bp(B_{i,m})$ converging to $0$ as $m \to \infty$ for each $1 \leq i \leq k$. This limit is independent of the choice of the $B_{i,m}$.
When $V=X$, we set
$$
\langle L_1 \cdots L_k \rangle := \langle L_1 \cdots L_k \rangle_{X|X}.
$$
By \cite[Proposition 4.13]{lehmann-nu}, we have
$$
\langle L_1 \cdots L_k \rangle|_{X|V} = \langle P_{\sigma}(L_1) \cdots P_{\sigma}(L_k) \rangle|_{X|V}.
$$
If $D$ is a big divisor on $X$ with $V \not\subseteq \bp(D)$, then $\deg \langle D^v \rangle|_{X|V} = \vol_{X|V}(D)$ (see \cite[Theorem 2.13]{elmnp2}).
For more details, we refer to \cite{BFJ}, \cite{lehmann-nu}.

\subsection{Okounkov bodies}
Let $Y_\bullet$ be an \emph{admissible flag} on $X$, which by definition is a sequence of irreducible subvarieties
$$
Y_\bullet: X=Y_0\supseteq Y_1\supseteq\cdots \supseteq Y_{n-1}\supseteq Y_n=\{x\}
$$
where each $Y_i$ has codimension $i$ in $X$ and is smooth at the point $x$. Assume that $|D|_{\R}:=\{ D' \mid D \sim_{\R} D' \geq 0 \}\neq\emptyset$.  For any $D'\in |D|_\R$, we define $\nu_1=\nu_1(D'):=\ord_{Y_1}(D')$. Then $D'-\nu_1 Y_1$ is effective and does not contain $Y_1$ in the support, so we can define $\nu_2=\nu_2(D'):=\ord_{Y_2}((D'-\nu_1Y_1)|_{Y_1})$. For $2 \leq i \leq n-1$, we inductively define $\nu_{i+1}=\nu_{i+1}(D'):=\ord_{Y_{i+1}}((\cdots((D'-\nu_1Y_1)|_{Y_1}-\nu_2Y_2)|_{Y_2}-\cdots-\nu_iY_i)|_{Y_{i}})$. By collecting the $\nu_i$, we obtain
$$
\nu_{Y_\bullet}(D'):=(\nu_1(D'),\nu_2(D'),\ldots,\nu_n(D')) \in \R^n_{\geq 0}.
$$
This defines a valuation-like function
$$
\nu_{Y_\bullet} \colon |D|_{\R}\to \R_{\geq0}^n.
$$
The \emph{Okounkov body} of $D$ with respect to $Y_\bullet$ is a convex subset of $\R^n$ defined as
$$
 \okbd_{Y_\bullet}(D):=\text{the closure of the convex hull of $\nu_{Y_\bullet}(|D|_{\R})$ in $\R^n_{\geq 0}$}.
$$
If $|D|_{\R} = \emptyset$, then we set $\okbd_{Y_\bullet}(D) := \emptyset$. If $D$ is big, then by \cite[Theorem A]{lm-nobody} we have
$$
\vol_{\R^n}(\okbd_{Y_\bullet}(D)) = \frac{1}{n!}\vol_X(D)~~\text{for every admissible flag $Y_\bullet$ on $X$}.
$$
The \emph{limiting Okounkov body} of $D$ with respect to $Y_\bullet$ is defined as
$$
\oklim_{Y_\bullet}(D):=\lim_{\varepsilon \to 0+} \okbd_{Y_\bullet}(D+\varepsilon A) = \bigcap_{\varepsilon > 0} \okbd_{Y_\bullet}(D+\varepsilon A).
$$
It is independent of the choice of the $A$, and $\oklim_{Y_\bullet}(D)=\okbd_{Y_\bullet}(D)$ when $D$ is big.
The limiting Okounkov bodies are numerical in nature; $D, D'$ are numerically equivalent pseudoeffective divisors if and only if $\oklim_{Y_\bullet}(D)=\oklim_{Y_\bullet}(D')$ for every admissible flag $Y_\bullet$.
It is easy to show that $\oklim_{Y_\bullet}(D)=\oklim_{Y_\bullet}(P_{\sigma}(D)) + \okbd_{Y_\bullet}(N_{\sigma}(D))$ when $D$ is pseudoeffective (cf. \cite[Lemma 3.3]{CPW2}).
For more details, we refer to \cite{CHPW}, \cite{KK}, \cite{lm-nobody}.

\subsection{Numerical dominance}
We say that $D$ \emph{numerically dominates} $V$ if there exists a birational morphism $\mu \colon W \to X$ with $\mu^{-1}\mathcal{I}_{V|X} \cdot \mathcal{O}_{W} = \mathcal{O}_W(E_V)$ for some effective divisor $E_V$ on $W$ such that for every positive number $b$ and every ample divisor $A$ on $W$, the divisor $x\mu^*D - yE_V+A$ is pseudoeffective for some $x,y > b$. For more details, we refer to \cite[Chapter 5]{nakayama}.

%%%%%%%%%%%%%%%%%%%%%%%%%%%%%%%%%%%%%%%%%%%%%%%%%%%%%
\section{Properties of numerical Iitaka dimensions}\label{sec:numiitakadim}
%%%%%%%%%%%%%%%%%%%%%%%%%%%%%%%%%%%%%%%%%%%%%%%%%%%%%

In this section, we study some properties of numerical Iitaka dimensions, and we prove Theorem \ref{thm:main} and Corollary \ref{cor:minmax}. By \cite[Proposition V.2.7 (1)]{nakayama}, $\kappa_{\sigma}(D)=\kappa_{\sigma, \sup}(D)$ depends only on the numerical equivalence class of $D$. The same proof shows that $\kappa_{\sigma, \inf}(D), \kappa_{\sigma, \inf}^{\R}(D), \kappa_{\sigma, \sup}^{\R}(D)$ are also numerical invariant of $D$. It is obvious from the definitions that all other numerical Iitaka dimensions considered in this paper are numerical invariants as well. Furthermore, it is easy to check that all the numerical Iitaka dimensions studied in this paper are birational invariant (see e.g. \cite[Propositions V.2.7 (4), V.2.22 (4)]{nakayama}).

\medskip

Let $X$ be a $\Q$-factorial normal projective variety, and $D$ be a pseudoeffective $\R$-divisor on $X$. By definition, we have $\nu_{\BDPP}(D) \geq 0$. It is also elementary to see that $\kappa(D) \leq \nu_{\BDPP}(D)$. In \cite{lehmann-nu}, Lehmann established that
\begin{equation}\label{eq:lehmann}
\nu_{\BDPP}(D)=\kappa_{\vol,\res}(D)=\kappa_{\vol,\zar}(D) \leq \kappa_{\sigma, \inf}(D) \text{ and } \kappa_{\vol,\zar}(D)\leq \kappa_{\vol, \inf}(D)
\end{equation}
(see the proof of \cite[Theorem 6.2]{lehmann-nu} for $(1)=(2)=(3)=(4)$, $(4) \leq (5)$, $(1) \leq (7)$). The inequality $\nu_{\BDPP}(D) \leq \kappa_{\nu}(D)$ was once claimed to be actually an equality by \cite{lehmann-nu}, \cite{E}. However, as was pointed out by Lesieutre \cite[Remark 4]{lesieutre}, the proof in \cite{lehmann-nu}, \cite{E} for this equality contains a serious gap. In fact, the strict inequality $\nu_{\BDPP}(D) < \kappa_{\nu}(D)$ can indeed hold (see \cite[Theorem 3]{lesieutre}). One of our goals in this paper is to clarify which numerical Iitaka dimensions coincide with each other and how other numerical Iitaka dimensions differ from the rest. First, we show the following:

\begin{proposition}\label{prop:ineqnumdim}
Let $X$ be a $\Q$-factorial normal projective variety, and $D$ be a pseudoeffective $\R$-divisor on $X$. Then we have
$$
\begin{array}{ccccccc}
&&\kappa_{\vol, \inf}^{\R}(D) & \leq & \kappa_{\vol, \sup}^{\R}(D) & \leq & \kappa_{\nu}(D)\\
&& \vleq & & \vleq &&\\
\nu_{\BDPP}(D) & \leq & \kappa_{\sigma, \inf}^{\R}(D) & \leq & \kappa_{\sigma, \sup}^{\R}(D).&&
\end{array}
$$
Similarly, we also have
$$
\begin{array}{ccccccc}
&&\kappa_{\vol, \inf}(D) & \leq & \kappa_{\vol, \sup}(D) & \leq & \kappa_{\nu}(D)\\
&& \vleq & & \vleq &&\\
\nu_{\BDPP}(D) & \leq & \kappa_{\sigma, \inf}(D) & \leq & \kappa_{\sigma, \sup}(D).&&
\end{array}
$$
\end{proposition}

If $D$ is nef, then $\kappa_{\nu}(D) = \nu(D)$ by \cite[Proposition V.2.22 (5)]{nakayama}. Furthermore, in this case, Theorem \ref{thm:main} and Proposition \ref{prop:ineqnumdim} imply that all the numerical Iitaka dimensions of $D$ considered in this paper are equal to $\nu(D)$.

\medskip

To prove Proposition \ref{prop:ineqnumdim}, we need some lemmas. The following lemma shows that our definition of $\kappa_{\vol}(D)$ is the same as the definition in \cite{lehmann-nu}.

\begin{lemma}\label{lem:kappavol}
Let $X$ be a $\Q$-factorial normal projective variety of dimension $n$, and $D$ be an $\R$-divisor on $X$. Fix a sufficiently positive ample $\Z$-divisor $A$ on $X$. Then for any fixed real number $k \geq 0$, the following are equivalent:
\begin{enumerate}
 \item[$(1)$] There is a constant $C>0$ such that $\vol_X(D+tA) > Ct^{n-k}$ for all $t > 0$.
 \item[$(2)$] $\liminf\limits_{\varepsilon \to 0+} \frac{\vol_X(D+\varepsilon A)}{\varepsilon^{n-k}}>0$.
 \item[$(3)$] $\liminf\limits_{m \to \infty} \frac{\vol_X(mD+ A)}{m^k}>0$.
\end{enumerate}
\end{lemma}

\begin{proof}
$(1) \Rightarrow (2)$: Trivial.

\medskip

\noindent $(2) \Rightarrow (3)$: Let us assume that $(2)$ holds. From the following observation
$$
\liminf\limits_{m \to \infty}  \frac{\vol_X(mD+ A)}{m^k} =\liminf\limits_{m \to \infty} \frac{m^n\vol_X\left(D+\frac{1}{m}A\right)}{m^k} = \liminf\limits_{m \to \infty} \frac{\vol_X \left(D+\frac{1}{m}A\right)}{\left(\frac{1}{m} \right)^{n-k}} > 0,
$$
we obtain $(3)$.

\medskip

\noindent $(3) \Rightarrow (1)$: Finally, let us assume that $(3)$ holds. There are constants $C', c'>0$ such that $\vol_X(mD+A) > c' m^k$ and $c' (m-1)^{n-k} > C' m^{n-k}$ for any sufficiently large integer $m> 0$. 
Fix a sufficiently small number $\delta > 0$. For any $t$ with $0<t<\delta$, there exists an integer $m$ such that $1/m \leq t < 1/(m-1)$. As $t$ is sufficiently small, we may assume that $m$ is sufficiently large.
Then we have
$$
\vol_X(D+tA) \geq \vol_X\left(D + \frac{1}{m}A \right) = \frac{\vol_X(mD+A)}{m^n} >  \frac{c'}{m^{n-k}} >  \frac{C'}{(m-1)^{n-k}} > C' t^{n-k}.
$$
Next, if $\delta \leq t < 1$, then
$$
\vol_X(D+tA) \geq \vol_X(D+\delta A) > \vol_X(D+\delta A) t^{n-k}.
$$
Fix a sufficiently small number $\varepsilon > 0$.
If $t \geq1$, then
$$
\vol_X(D+tA) = t^n \vol_X \left( \frac{1}{t}D + A \right) > t^n \vol_X((1-\varepsilon) A) \geq \vol_X((1-\varepsilon)A) t^{n-k}.
$$
By taking $C:= \min \{ C', \vol_X(D+\delta A), \vol_X((1-\varepsilon) A ) \}$, we obtain $(1)$.
\end{proof}

\begin{lemma}\label{lem:vol>h^0}
Let $X$ be a $\Q$-factorial normal projective variety, $D$ be a pseudoeffective $\R$-divisor on $X$, and $A$ be a sufficiently positive ample $\Z$-divisor on $X$. Then we have
$$
\vol_{X}(mD+2A) \geq h^0(X, \lfloor mD \rfloor +A)~~\text{for any integer $m > 0$}.
$$
\end{lemma}

\begin{proof}
Let $n:=\dim X$, and fix an admissible flag $Y_\bullet$ on $X$. Note that $\okbd_{Y_\bullet}(A)$ contains a standard simplex $\Delta$ of size $1$ in $\mathbb R^n$.
For each integer $m \geq 1$, let $\Delta_m:=\okbd_{Y_\bullet}(mD+A)$. By \cite[Lemma 1.4]{lm-nobody}, we have
\begin{equation*}\label{eq:num>=h0}
\# (\Delta_m \cap \Z^n) \geq h^0(X, \lfloor mD \rfloor +A).
\end{equation*}
Notice now that
$$
\Delta_m + \Delta \supseteq \bigcup_{\mathbf{x} \in \Delta_m \cap \Z^n} \left(\{\mathbf{x}\}+\Delta \right).
$$
As the interiors of $\{\mathbf{x}\}+\Delta$ are disjoint and $\vol_{\R^n}(\{\mathbf{x}\}+\Delta)=\vol_{\R^n}(\Delta)=1/n!$, we get
$$
n! \cdot \vol_{\R^n} (\Delta_m + \Delta) \geq \# (\Delta_m \cap \Z^n).
$$
Note that $\Delta_{Y_\bullet}(mD+2A)\supseteq\Delta_{Y_\bullet}(mD+A)+\Delta_{Y_\bullet}(A)\supseteq\Delta_m+\Delta$ and $\vol_X(mD+2A)=n!\cdot \vol_{\R^n}\big(\Delta_{Y_\bullet}(mD+2A))$. Thus we obtain
$$
\vol_X(mD+2A) \geq n! \cdot \vol_{\R^n}(\Delta_m + \Delta) \geq \# (\Delta_m \cap \Z^n) \geq  h^0(X, \lfloor mD \rfloor +A).
$$
We finish the proof.
\end{proof}

We now prove Proposition \ref{prop:ineqnumdim}.

\begin{proof}[Proof of Proposition \ref{prop:ineqnumdim}]
Let $n:=\dim X$. In this proof, $\star = \sigma$ or $\vol$, and $\diamond = \inf$ or $\sup$. By definition, we have
$$
\kappa_{\star, \diamond}(D) \leq \kappa_{\star, \diamond}^{\R}(D),~\kappa_{\star, \inf}^{\R}(D) \leq \kappa_{\star, \sup}^{\R}(D),~\kappa_{\star, \inf}(D) \leq \kappa_{\star, \sup}(D).
$$
Thus it is sufficient to show the following four inequalities:
$$
\nu_{\BDPP}(D) \leq \kappa_{\sigma, \inf}(D),~\kappa_{\sigma, \diamond}^{\R}(D) \leq \kappa_{\vol, \diamond}^{\R}(D),~\kappa_{\sigma, \diamond}(D) \leq \kappa_{\vol, \diamond}(D), ~\kappa_{\vol, \sup}^{\R}(D) \leq \kappa_{\nu}(D).
$$

\medskip

\noindent  $\text{\tiny$\bullet$}~ \nu_{\BDPP}(D) \leq \kappa_{\sigma, \inf}(D)$: This is established in (\ref{eq:lehmann}) by \cite{lehmann-nu}.

\medskip

\noindent $\text{\tiny$\bullet$}~\kappa_{\sigma, \diamond}^{\R}(D) \leq \kappa_{\vol, \diamond}^{\R}(D),~\kappa_{\sigma, \diamond}(D) \leq \kappa_{\vol, \diamond}(D)$:
Let $A$ be a sufficiently positive ample $\Z$-divisor on $X$. By Lemma \ref{lem:vol>h^0}, we have
$$
\frac{\vol_{X}(mD+2A)}{m^k}  \geq \frac{h^0(X, \lfloor mD \rfloor +A)}{m^k}
$$
for any integer $m > 0$ and any real number $k\geq 0$. Then Lemma \ref{lem:kappavol} implies that $\kappa_{\sigma, \inf}^{\R}(D) \leq \kappa_{\vol, \inf}^{\R}(D)$ and $\kappa_{\sigma, \inf}(D) \leq \kappa_{\vol, \inf}(D)$.
Now, we observe that
$$
\frac{\vol_X\left(D + \frac{2}{m}A \right)}{\left(\frac{1}{m} \right)^{n-k}} = \frac{m^n \vol_X\left( D+ \frac{2}{m} A \right)}{m^k} = \frac{\vol_{X}(mD+2A)}{m^k}.
$$
Then we find
$$
\limsup_{\varepsilon \to 0} \frac{\vol_X(D+\varepsilon\cdot 2A)}{\varepsilon^{n-k}} \geq \limsup_{m \to \infty} \frac{\vol_X\left(D + \frac{2}{m}A \right)}{\left(\frac{1}{m} \right)^{n-k}} =  \limsup_{m \to \infty}\frac{\vol_{X}(mD+2A)}{m^k}.
$$
This implies that $\kappa_{\sigma, \sup}^{\R}(D) \leq \kappa_{\vol, \sup}^{\R}(D)$ and $\kappa_{\sigma, \sup}(D) \leq \kappa_{\vol, \sup}(D)$.

\medskip

\noindent $\text{\tiny$\bullet$}~ \kappa_{\vol, \sup}^{\R}(D) \leq \kappa_{\nu}(D)$:
We closely follow \cite[Proof of Proposition V.2.22 (1)]{nakayama}.
Take any nonnegative integer $w$ with $w < \kappa_{\vol, \sup}^{\R}(D)$.
Let $A$ be a very ample divisor on $X$, and $W \subseteq X$ be a complete intersection of $(n-w)$ general members of $|A|$. Denote by $\mathcal{I}_W$ the ideal sheaf of $W$ in $X$. Note that $N_{W/X}^* \cong \mathcal{O}_W(-A)^{\oplus n-w}$. Fix an integer $b > 0$.
For any integer $m>0 $ and real number $x > 0$, we have an exact sequence
$$
\begin{array}{l}
0 \to H^0(X, \mathcal{I}_W^{mb+1} \cdot \mathcal{O}_X(\lfloor mxD \rfloor + mA) \to H^0(X, \mathcal{I}_W^{mb} \cdot \mathcal{O}_X(\lfloor mxD \rfloor +mA)\\
 \to H^0(W, S^{mb}(N_{W/X}^*) \otimes \mathcal{O}_W(\lfloor mxD \rfloor +mA) ).
\end{array}
$$
Then we get
\begin{footnotesize}
$$
h^0(X, \mathcal{I}_W^{mb+1} \cdot \mathcal{O}_X(\lfloor mxD \rfloor + mA)) \geq h^0(X, \lfloor mxD \rfloor + mA) - {n-w+mb \choose n-w} h^0(W, (\lfloor mxD \rfloor + mA)|_W).
$$
\end{footnotesize}\\[-10pt]
Now, fix a real number $k$ with $w < k < \kappa_{\vol, \sup}^{\R}(D)$. There exists a constant $C>0$ such that $\vol_X(xD+A) > C x^k$ for every $x \in S$, where $S$ is a set of some sufficiently large positive real numbers with $\sup S = \infty$. Then there is a constant $C'>0$ such that
$$
h^0(X, \lfloor mxD \rfloor + mA) > C' m^n x^k.
$$
for every sufficiently large integer $m>0$ and $x \in S$. On the other hand, there exists a constant $C'' > 0$ such that
$$
{n-w+mb \choose n-w} h^0(W, (\lfloor mxD \rfloor + mA)|_W) < C'' (mb)^{n-w} (mx)^w = (C'' b^{n-w}) m^n x^w,
$$
for every sufficiently large integer $m>0$ and $x \in S$. Thus we obtain $h^0(X, \mathcal{I}_W^{mb+1} \cdot \mathcal{O}_X(\lfloor mxD \rfloor + mA)) > 0$. This means that $D$ numerically dominates $W$. Hence $\kappa_{\nu}(D) > w$, and therefore, $\kappa_{\nu}(D) \geq \lceil \kappa_{\vol, \sup}^{\R}(D) \rceil \geq \kappa_{\vol, \sup}^{\R}(D)$.
\end{proof}

By the main example in \cite[Section 3]{lesieutre}, the strict inequalities $\nu_{\BDPP}(D) < \kappa_{\sigma, \inf}^{\R}(D)$ and $\kappa_{\vol, \sup}^{\R}(D) < \kappa_{\nu}(D)$ can indeed hold. It is natural to ask the following.

\begin{question}[{cf. \cite[Remark 9]{lesieutre}}]\label{ques:sigma=vol?}
Let $X$ be a $\Q$-factorial normal projective variety, and $D$ be a pseudoeffective $\R$-divisor on $X$. Do the following equalities hold?:
\begin{small}
\begin{equation*}\label{eq:sigma=vol2}
\kappa_{\sigma, \inf}^{\R}(D) = \kappa_{\vol, \inf}^{\R}(D),~\kappa_{\sigma, \sup}^{\R}(D)=\kappa_{\vol, \sup}^{\R}(D),~\kappa_{\sigma, \inf}(D) = \kappa_{\vol, \inf}(D),~\kappa_{\sigma, \sup}(D)=\kappa_{\vol, \sup}(D).
\end{equation*}
\end{small}
\end{question}

\begin{remark}\label{rem:newcounterex}
We are informed from John Lesieutre that he recently constructed an example of $\kappa_{\sigma, \inf}^{\R}(D) \neq \kappa_{\sigma, \sup}^{\R}(D)$.
\end{remark}

Next, we turn to the proof of Theorem \ref{thm:main}, which claims the following equalities:
$$
\nu_{\BDPP}(D)=\kappa_{\vol, \res}(D)=\kappa_{\vol,\zar}(D)= \kappa_{\vol, \inf, \bir}^{\R}(D) = \kappa_{\vol, \sup, \bir}^{\R}(D)  = \kappa_{\Delta}(D).
$$

\begin{proof}[Proof of Theorem \ref{thm:main}]
Let $X$ be a $\Q$-factorial normal projective variety of dimension $n$, and $D$ be a pseudoeffective $\R$-divisor on $X$.
Recall from (\ref{eq:lehmann}) that the equalities $\nu_{\BDPP}(D)=\kappa_{\vol,\res}(D)=\kappa_{\vol,\zar}(D)$ were shown by Lehmann. By definition, $\kappa_{\vol, \inf, \bir}^{\R}(D) \leq \kappa_{\vol, \sup, \bir}^{\R}(D)$ holds. We will first show that $$\nu_{\BDPP}(D) \leq \kappa_{\vol, \inf, \bir}^{\R}(D) \text{ and } \kappa_{\vol, \sup, \bir}^{\R}(D) \leq \kappa_{\vol, \res}(D),$$
thereby showing that the first five numerical Iitaka dimensions coincide:
$$
\nu_{\BDPP}(D)=\kappa_{\vol, \res}(D)=\kappa_{\vol,\zar}(D)=  \kappa_{\vol, \inf, \bir}^{\R}(D)=\kappa_{\vol, \sup, \bir}^{\R}(D).
$$
The inequality $\kappa_{\vol, \res}(D) \leq \kappa_{\Delta}(D)$ was shown in \cite[Proof of Proposition 3.21]{CHPW}. We complete the proof finally by showing that $\kappa_{\Delta}(D) \leq \kappa_{\vol, \inf, \bir}^{\R}(D)$.

\medskip

\noindent $\text{\tiny$\bullet$}~ \nu_{\BDPP}(D) \leq \kappa_{\vol, \inf, \bir}^{\R}(D)$:
This inequality can be similarly shown as in the proof of \cite[Theorem 6.2]{lehmann-nu} for $(7) \leq (1)$, but we present the detailed proof for readers' convenience. Let $k:=\nu_{\BDPP}(D)$.
Fix a sufficiently positive ample $\Z$-divisor $A$ on $X$, and let $\phi \colon \widetilde{X} \to X$ be a birational morphism with $\widetilde{X}$ smooth projective. Then there exists a constant $C>0$ such that $C<\langle (P_{\sigma}(\phi^*D) + \varepsilon \phi^*A)^k \rangle \cdot \phi^*A^{n-k}$ for every $\varepsilon > 0$. We have
\begin{equation}\label{eq:vol,birlower}
\begin{array}{rcl}
C\varepsilon^{n-k} &<& \varepsilon^{n-k} \langle (P_{\sigma}(\phi^*D) + \varepsilon \phi^*A)^k \rangle \cdot \phi^*A^{n-k}\\
& =& \langle (P_{\sigma}(\phi^*D) + \varepsilon \phi^*A)^k \cdot (\varepsilon \phi^*A)^{n-k} \rangle \\
&\leq& \langle (P_{\sigma}(\phi^*D) + \varepsilon \phi^*A)^n \rangle\\
&=& \vol_{\widetilde{X}}(P_{\sigma}(\phi^*D) + \varepsilon \phi^*A)
\end{array}
\end{equation}
for every $\varepsilon > 0$. Thus the inequality $\kappa_{\BDPP}(D)=k \leq \kappa_{\vol, \inf, \bir}^{\R}(D)$ holds.

\medskip

\noindent $\text{\tiny$\bullet$}~\kappa_{\vol, \sup, \bir}^{\R}(D) \leq \kappa_{\vol, \res}(D)$:
Let $k:=\kappa_{\vol, \res}(D)$.
By \cite[Proposition 3.8]{lehmann-nu}, for each integer $m \geq 1$, there are a birational morphism
$$
\phi_m \colon X_m \to X
$$
centered in $\bm(D)$, a nef and big divisor $M_m$ on $X_m$, an ample $\Z$-divisor $H$ on $X$, and an effective divisor $G$ on $X$ such that
$$
M_m \leq \frac{1}{m}P_{\sigma}(\phi_m^*(\lceil mD \rceil + H)) \leq M_m + \frac{1}{m}\phi_m^*G.
$$
Fix a sufficiently positive very ample $\Z$-divisor $A$ on $X$. For any number $\varepsilon > 0$, we have
$$
P_{\sigma}(\phi_m^*D) + \varepsilon \phi_m^* A \leq \frac{1}{m}P_{\sigma}(\phi_m^*(\lceil mD \rceil + H)) + \varepsilon \phi_m^*A \leq M_m + \frac{1}{m}\phi_m^*G + \varepsilon \phi_m^*A,
$$
thus we obtain
\begin{equation}\label{eq:volbir}
\begin{array}{rcl}
\displaystyle \lim_{m \to \infty} \vol_{X_m} (P_{\sigma}(\phi_m^*D) + \varepsilon \phi_m^* A )& \leq & \displaystyle \lim_{m \to \infty}  \vol_{X_m} \left(M_m + \frac{1}{m}\phi_m^*G + \varepsilon \phi_m^*A \right) \\
&=& \displaystyle \lim_{m \to \infty} \vol_{X_m} (M_m + \varepsilon \phi_m^*A) \\
&=&  \displaystyle \lim_{m \to \infty} (M_m+\varepsilon \phi_m^* A)^n \\
    &=& \displaystyle \lim_{m \to \infty} \sum_{i=0}^n {n \choose i} \varepsilon^i M_m^{n-i} (\phi_m^* A)^i .
\end{array}
\end{equation}
Let $W_i$ be a complete intersection of general $i$ members in $|A|$ for $0 \leq i \leq n$, and $W_{i,m}$ be the strict transform via $\phi_m$. Then
\begin{align*}%{rl}
M_m^{n-i} (\phi_m^* A)^i &= \vol_{X_m|W_{i,m}}(M_m)\\
 &\leq  \vol_{X_m|W_{i,m}}\left(\frac{1}{m}P_{\sigma}(\phi_m^*(\lceil mD \rceil + H)) \right)\\
 & = \vol_{X|W_i}\left( \frac{1}{m} (\lceil mD \rceil + H) \right).
 \end{align*}
Notice that
$$
\lim_{m \to \infty}  \vol_{X|W_i}\left( \frac{1}{m} (\lceil mD \rceil + H) \right) = \vol_{X|W_i}^+(D).
$$
Thus we get
$$
\lim_{m \to \infty} M_m^{n-i} (\phi_m^* A)^i  = 0~~\text{when $n-i = \dim W_i \geq k+1= \kappa_{\vol, \res}(D)+1$}.
$$
From (\ref{eq:volbir}), we obtain
\begin{equation}\label{eq:vol,birupper}
 \begin{array}{rl}
\displaystyle \lim_{m \to \infty} \vol_{X_m} (P_{\sigma}(\phi_m^*D) + \varepsilon \phi_m^* A ) &\leq \varepsilon^{n-k} \displaystyle\sum_{i =0}^n {n \choose i} \varepsilon^{i-n+k} M_m^{n-i} (\phi_m^* A)^i\\
 &=\varepsilon^{n-k} \displaystyle\sum_{i = n-k}^n {n \choose i} \varepsilon^{i-n+k} M_m^{n-i} (\phi_m^* A)^i
 \end{array}
\end{equation}
This shows that
$$
\limsup_{\varepsilon \to 0}  \lim_{m \to \infty} \frac{ \vol_{X_m} (P_{\sigma}(\phi_m^*D) + \varepsilon \phi_m^* A )}{\varepsilon^{n-k-\delta}} = 0
$$
for every $\delta > 0$. Hence $\kappa_{\vol, \sup, \bir}^{\R}(D) \leq k=\kappa_{\vol,\res}(D)$.

\medskip

\noindent $\text{\tiny$\bullet$}~ \kappa_{\Delta}(D) \leq \kappa_{\vol, \inf, \bir}^{\R}(D)$:
By definition, there exists an admissible flag $Y_\bullet$ on $X$ such that $k:=\kappa_{\Delta}(D)=\dim \oklim_{Y_\bullet}(D)$. It is enough to prove that
\begin{equation}\label{eq:oklimcomp}
\frac{1}{n!} \vol_{\widetilde{X}}(P_{\sigma}(\phi^*D) + \varepsilon \phi^*A)  \geq \vol_{\R^k}(\oklim_{Y_\bullet}(D)) \cdot \varepsilon^{n-k}
\end{equation}
for any sufficiently small number $\varepsilon > 0$, where $A$ is a fixed sufficiently positive ample $\Z$-divisor and $\phi \colon \widetilde{X} \to X$ is an arbitrary birational morphism $\phi \colon \widetilde{X} \to X$ from a smooth projective variety $\widetilde{X}$.
By taking sufficient blow-ups, we obtain a $Y_\bullet$-admissible birational morphism $\phi$ in the sense of \cite[Definition 3.5]{CPW2}. This simply implies that we can form an admissible flag $\widetilde{Y}_\bullet$ on $\widetilde{X}$ by taking proper transforms $\widetilde{Y}_i$ of each $Y_i$ of $Y_\bullet$ so that $\widetilde{Y}_\bullet$ is an induced proper admissible flag in the sense of \cite[Definition 3.4]{CPW2}. By \cite[Lemma 3.7]{CPW2}, we get
$$
\dim \oklim_{\widetilde{Y}_\bullet}(\phi^* D) = \dim \oklim_{Y_\bullet}(D)=k~~\text{and}~~\vol_{\R^k}\big(\oklim_{\widetilde{Y}_\bullet}(\phi^*D)\big)  = \vol_{\R^k}(\oklim_{Y_\bullet}(D)).
$$
As $\oklim_{\widetilde{Y}_\bullet}(\phi^*D) = \oklim_{\widetilde{Y}_\bullet}(P_{\sigma}(\phi^*D))$, the subadditivity of Okounkov bodies shows that
$$
\okbd_{\widetilde{Y}_\bullet}(P_{\sigma}(\phi^*D)+\varepsilon \phi^*A) \supseteq \oklim_{\widetilde{Y}_\bullet}(P_{\sigma}(\phi^*D)) +  \okbd_{\widetilde{Y}_\bullet}(\varepsilon \phi^*A) = \oklim_{\widetilde{Y}_\bullet}(\phi^*D) +\varepsilon\okbd_{\widetilde{Y}_\bullet}( \phi^*A).
$$
Since $A$ is sufficiently positive, it follows that $\okbd_{\widetilde{Y}_\bullet}( \phi^*A)$ contains a sufficiently large convex body. Thus we obtain
\begin{align*}
\vol_{\R^n}(\okbd_{\widetilde{Y}_\bullet}(P_{\sigma}(\phi^*D)+\varepsilon \phi^*A)  &\geq \vol_{\R^n}( \oklim_{\widetilde{Y}_\bullet}(\phi^*D) +\varepsilon\okbd_{\widetilde{Y}_\bullet}( \phi^*A)) \\
&\geq \vol_{\R^k}(\oklim_{Y_\bullet}(D)) \cdot \varepsilon^{n-k}
\end{align*}
which implies (\ref{eq:oklimcomp}). Hence $\kappa_{\vol, \inf, \bir}^{\R}(D) \geq k=\kappa_{\Delta}(D)$.
\end{proof}

\begin{remark}
In \cite{CHPW, CPW}, the notation $\kappa_{\nu}$ was used to denote the numerical Iitaka dimension $\kappa_{\sigma}$. All the results in \cite{CHPW, CPW} hold by replacing $\kappa_{\nu}$ by $\nu_{\BDPP}$.
\end{remark}

It would be also interesting to know whether the fundamental inequalities (\ref{eq:iitakadim}) for Iitaka dimensions have analogues for numerical Iitaka dimensions.

\begin{question}[{cf. \cite[Conjecture 1.4]{F}, \cite[Remark 5]{lesieutre}, \cite[Question 4.1]{M}}]\label{ques:fundineq}
Let $X$ be a $\Q$-factorial normal projective variety, and $D$ be a pseudoeffective $\R$-divisor on $X$.\\[3pt]
$(1)$ Do there exist constants $C_1, C_2 > 0$ such that
$$
C_1 m^{\kappa_{\sigma, \inf}^{\R}(D)} < h^0(X, \lfloor mD \rfloor + A) < C_2 m^{\kappa_{\sigma, \sup}^{\R}(D)}
$$
for every sufficiently large integer $m > 0$?\\[3pt]
$(2)$ Do there exist constants $C_1, C_2 > 0$ such that
$$
C_1 \varepsilon^{n-\kappa_{\vol, \inf}^{\R}(D)} < \vol_X (D+\varepsilon A) < C_2 \varepsilon^{n-\kappa_{\vol, \sup}^{\R}(D)}
$$
for every sufficiently small number $\varepsilon > 0$?
\end{question}

At this moment, as stated in Corollary \ref{cor:minmax}, we can only show that the inequalities in $(2)$ hold on the limit of higher birational models.

\begin{proof}[Proof of Corollary \ref{cor:minmax}]
The existence of lower bound follows from  (\ref{eq:vol,birlower})  in the proof of Theorem \ref{thm:main}. For a given birational morphism $\phi \colon \widetilde{X} \to X$, we may replace $\phi$ by a common resolution of $\phi$ and $\phi_m$ in  the proof of $\kappa_{\vol, \sup, \bir}^{\R}(D) \leq \kappa_{\vol, \res}(D)$ in Theorem \ref{thm:main}. Thus the existence of the upper bound follows similarly from (\ref{eq:vol,birupper}) in the same proof.
\end{proof}

Let $X$ be a $\Q$-factorial normal projective variety, and $D$ be a pseudoeffective $\R$-divisor on $X$.  It is highly desirable that the positive part $P_\sigma(D)$ in the divisorial Zariski decomposition of $D$ computes each of the numerical Iitaka dimension of $D$. It is easy to see that
$$
\nu_{\BDPP}(D)=\nu_{\BDPP}(P_{\sigma}(D)).
$$
Although expected, this property is unknown for other numerical Iitaka dimensions.

\begin{question}[{cf. \cite[Problem in p.165]{nakayama}}]\label{ques:numdimdivzar}
Let $X$ be a $\Q$-factorial normal projective variety, and $D$ be a pseudoeffective $\R$-divisor on $X$. Do the following equalities hold?
$$
\kappa_{\star, \diamond}^{\R}(D)=\kappa_{\star, \diamond}^{\R}(P_{\sigma}(D)),~\kappa_{\star, \diamond}(D)=\kappa_{\star, \diamond}(P_{\sigma}(D)),~\kappa_{\nu}(D) = \kappa_{\nu}(P_{\sigma}(D)),
$$
where $\star= \sigma$ or $\vol$ and $\diamond = \inf$ or $\sup$.
\end{question}

As was remarked by Nakayama \cite[p.165]{nakayama}, the affirmative answer to the following question gives a positive answer to Question \ref{ques:numdimdivzar}.

\begin{question}[{cf. \cite[Problem V.1.8]{nakayama}}]
Let $X$ be a $\Q$-factorial normal projective variety, and $D$ be a pseudoeffective $\R$-divisor on $X$. For an ample $\Z$-divisor $A$ on $X$ and a prime divisor $\Gamma$ on $X$, is $t \ord_{\Gamma}(||D||) - \ord_{\Gamma}(||tD+A||)$ bounded for $t>0$?
\end{question}

Finally, we study the cases where $\nu_{\BDPP}(D)=0,n-1,$ or $n$.

\begin{proposition}\label{prop:numdim=0,n-1,n}
Let $X$ be a $\Q$-factorial normal projective variety of dimension $n$, and $D$ be a pseudoeffective $\R$-divisor on $X$. Then we have the following:
\begin{enumerate}
 \item[$(1)$] $\nu_{\BDPP}(D)=0 ~\Leftrightarrow~ P_{\sigma}(D) \equiv 0 ~\Leftrightarrow~ \kappa_{\nu}(D)=0$.
 \item[$(2)$] $\nu_{\BDPP}(D) = n-1 ~\Leftrightarrow~ \kappa_{\vol, \sup}(D)=n-1$.
 \item[$(3)$] $\nu_{\BDPP}(D)=n~\Leftrightarrow~ \text{$D$ is big} ~\Leftrightarrow~ \kappa_{\nu}(D)=n$.\end{enumerate}
In particular, if any one of $\nu_{\BDPP}(D), \kappa_{\sigma}(D), \kappa_{\vol}(D)$ is $0, n-1$, or $n$, then the rest of them also coincide with the same number.
\end{proposition}

\begin{proof}
It is easy to check that $\nu_{\BDPP}(D)=0 ~\Leftrightarrow~ P_{\sigma}(D) \equiv 0$ and $\nu_{\BDPP}(D)=n~\Leftrightarrow~ \text{$D$ is big}$. By \cite[Proposition V.2.22 (2)]{nakayama}, $P_{\sigma}(D) \equiv 0 ~\Leftrightarrow~ \kappa_{\nu}(D)=0$, and by definition, $\text{$D$ is big} ~\Leftrightarrow~ \kappa_{\nu}(D)=n$. Thus $(1), (3)$ hold.

\medskip

We now show $(2)$. If $\nu_{\BDPP}(D) = n-1$, then $D$ is not big and thus $\kappa_{\nu}(D) \leq n-1$. As $\nu_{\BDPP}(D) \leq \kappa_{\vol, \sup}(D) \leq \kappa_{\nu}(D)$ by Proposition \ref{prop:ineqnumdim}, we obtain $\kappa_{\vol, \sup}(D)=n-1$. For the converse, we assume that $\kappa_{\vol, \sup}(D)=n-1$. Fix a sufficiently positive ample $\Z$-divisor $A$, and take a general member $H \in |A|$. For $t> 0$, let
$$
f(t):=\frac{\vol_X(D+tA)}{t}.
$$
By differentiability of volume function (\cite[Theorem A]{BFJ}, \cite[Corollary C]{lm-nobody}), we have
$$
f'(t)=\frac{n\vol_{X|H}(D+tA)t - \vol_X(D+tA)}{t^2}.
$$
Note first that $f(t)$ is a non-decreasing function for $t > 0$ since $f'(t) \geq 0$ for all $t > 0$ by \cite[Lemma 6.3]{lehmann-nu}.
As $\kappa_{\vol, \sup}(D)=n-1$, we get $\displaystyle \lim_{t \to 0} f(t) > 0$. By \cite[Corollary C]{BFJ}, we have
$$
 \vol_{X|H}^+(D) = \lim_{t \to 0}f(t) >0,
$$
and hence, we obtain $\nu_{\BDPP}(D) = \kappa_{\vol, \res}(D) = n-1$. We finish the proof of $(2)$.
Now, the last assertion follows from Proposition \ref{prop:ineqnumdim}.
\end{proof}

\begin{remark}
By Lesieutre's example in \cite{lesieutre}, it is known that there is a smooth projective threefold $Y$ and a pseudoeffective divisor $E$ such that $\nu_{\BDPP}(E)=1$ and $\kappa_{\nu}(E)=2$. Thus, in the setting of Proposition \ref{prop:numdim=0,n-1,n}, $\kappa_{\nu}(D)=n-1$ does not imply $\nu_{\BDPP}(D)=n-1$.
\end{remark}

%%%%%%%%%%%%%%%%%%%%%%%%%%%%%%%%%%%%%%%%%%%%%%%%%%%%%
\section{Example of $\kappa_{\sigma}(D), \kappa_{\vol}(D) > \nu_{\BDPP}(D)$}\label{sec:example}
%%%%%%%%%%%%%%%%%%%%%%%%%%%%%%%%%%%%%%%%%%%%%%%%%%%%%
We give a proof of Theorem \ref{thm:notcoincide} as a part of Example \ref{ex:sigma,vol>bdpp} using the following additivity for $\nu_{\BDPP}$.

\begin{proposition}\label{prop:product}
Let $X_1, X_2$ be smooth projective varieties, $p_i \colon X_1 \times X_2 \to X_i$ be the projection to the $i$-th component for $i=1,2$, and $D_1, D_2$ be pseudoeffective $\R$-divisors on $X_1, X_2$, respectively. Then we have
$$
\nu_{\BDPP}(p_1^*D_1 + p_2^*D_2) = \nu_{\BDPP}(D_1)+\nu_{\BDPP}(D_2).
$$
\end{proposition}

\begin{proof}
Let $k_1:=\nu_{\BDPP}(D_1)$ and $k_2:=\nu_{\BDPP}(D_2)$. For each $i=1,2$, there is a $k_i$-dimensional subvariety $W_i \subseteq X_i$ such that $\vol_{X_i|W_i}^+(D_i) > 0$. Note that
$$
\vol_{X_1 \times X_2 | W_1 \times W_2}^+(p_1^*D_1 + p_2^*D_2) = \vol_{X_1|W_1}^+(D_1) \cdot \vol_{X_2|W_2}^+(D_2).
$$
Thus $\nu_{\BDPP}(p_1^*D_1 + p_2^*D_2) \geq k_1 + k_2$.
To derive a contradiction, we now suppose that $\nu_{\BDPP}(p_1^*D_1 + p_2^*D_2) >  k_1+k_2$. Then there is a $(k_1+k_2+1)$-dimensional subvariety $W \subseteq X_1 \times X_2$ such that $\vol_{X_1 \times X_2 | W}^+(p_1^*D_1 + p_2^*D_2) > 0$. We may assume that $W$ is a complete intersection of sufficiently positive very ample divisors (see \cite[Proposition 2.22]{CHPW}). For each $i=1,2$, by \cite[Proposition 3.8]{lehmann-nu}, there are birational maps
$$
\phi_{i,m} \colon X_{i,m} \to X_i
$$
centered in $\bm(D_i)$ for every integer $m \geq 1$, nef and big divisors $M_{i,m}$ on $X_{i,m}$, an ample $\Z$-divisor $A_i$ on $X_i$, and an effective divisor $G_i$ on $X_i$ such that
$$
M_{i,m} \leq \frac{1}{m} P_{\sigma}(\phi_{i,m}^*(\lceil mD_i \rceil + A_i)) \leq M_{i,m} + \frac{1}{m} \phi_{i,m}^*G_i.
$$
Let
$$
\phi_m :=\phi_{1,m} \times \phi_{2,m} \colon X_{1,m} \times X_{2,m} \to X_1 \times X_2
$$
with the projection $p_{i,m} \colon X_{1,m} \times X_{2,m} \to X_{i,m}$ for each $i=1,2$, and $W_m\subseteq X_{1,m}\times X_{2,m}$ be the strict transform of $W\subseteq X_1\times X_2$ under $\phi_m$.
Notice that
$$
\begin{array}{l}
\displaystyle  \vol_{X_{1,m}\times X_{2,m}|W_m}\left( p_{1,m}^*\left(M_{1,m}+\frac{1}{m}\phi_{1,m}^*G_1 \right) + p_{2,m}^*\left( M_{2,m}+\frac{1}{m}\phi_{2,m}^*G_2 \right) \right)\\
\displaystyle \geq \vol_{X_1 \times X_2|W}\left(p_1^*\left( \frac{\lceil mD_1 \rceil}{m} + \frac{1}{m}A_1 \right)  + p_2^*\left(\frac{\lceil mD_2 \rceil}{m} + \frac{1}{m}A_2 \right)\right).\\
 \end{array}
$$
Since we have
\begin{footnotesize}
$$
\displaystyle  \lim_{m \to \infty} \vol_{X_1 \times X_2|W}\left(p_1^*\left( \frac{\lceil mD_1 \rceil}{m} + \frac{1}{m}A_1 \right)  + p_2^*\left(\frac{\lceil mD_2 \rceil}{m} + \frac{1}{m}A_2 \right)\right)= \vol_{X_1 \times X_2|W}(p_1^*D_1 + p_2^*D_2) > 0,
$$
\end{footnotesize}\\[-10pt]
it follows that
$$
\lim_{m \to \infty} \vol_{X_{1,m}\times X_{2,m}|W_m}(p_{1,m}^*M_{1,m} + p_{2,m}^*M_{2,m}) > 0.
$$
On the other hand, we have
$$
\begin{array}{l}
 \displaystyle\vol_{X_{1,m}\times X_{2,m}|W_m}(p_{1,m}^*M_{1,m} + p_{2,m}^*M_{2,m})= (p_{1,m}^*M_{1,m} + p_{2,m}^*M_{2,m})^{k_1+k_2+1} \cdot W_m\\
~~= \displaystyle \sum_{j=0}^{k_1+k_2+1} {k_1+k_2+1 \choose j} (p_{1,m}^*M_{1,m})^j \cdot (p_{2,m}^*M_{2,m})^{k_1+k_2+1-j} \cdot W_m.
\end{array}
$$
Note that $\displaystyle \lim_{m \to \infty} M_{1,m}^{k_1+1} = 0$ and $\displaystyle \lim_{m \to \infty} M_{2,m}^{k_2+1} = 0$.
For each $0 \leq j \leq k_1+k_2+1$, either $j \geq k_1+1$ or $k_1+k_2+1-j \geq k_2+1$.
Thus
$$
\lim_{m \to \infty} (p_{1,m}^*M_{1,m})^j \cdot (p_{2,m}^*M_{2,m})^{k_1+k_2+1-j} \cdot W_m = 0~~\text{for every $0 \leq j \leq k_1 + k_2 + 1$},
$$
which implies the following contradiction:
$$
\lim_{m \to \infty} \vol_{X_{1,m}\times X_{2,m}|W_m}(p_{1,m}^*M_{1,m} + p_{2,m}^*M_{2,m}) = 0.
$$
Therefore, $\nu_{\BDPP}(p_1^*D_1 + p_2^*D_2) = k_1 + k_2$.
\end{proof}

We give a proof of Theorem \ref{thm:notcoincide} based on the example in \cite{lesieutre}.

\begin{example}\label{ex:sigma,vol>bdpp}
In \cite{lesieutre}, Lesieutre proved that there exist a smooth projective threefold $Y$ and a pseudoeffective divisor $E$ on $Y$ such that for any sufficiently positive ample $\Z$-divisor $A$ on $X$, there are constants $C_1, C_2, C_1', C_2' > 0$ satisfying
$$
C_1 m^{\frac{3}{2}} < h^0(X, \lfloor mD \rfloor + A) < C_2 m^{\frac{3}{2}}~~\text{ and }~~C_1' \varepsilon^{\frac{3}{2}} < \vol_X (D+ \varepsilon A) < C_2' \varepsilon^{\frac{3}{2}}
$$
for any sufficiently large integer $m>0$ and any sufficiently small number $\varepsilon>0$. Then
$$
\nu_{\BDPP}(E)=1~~\text{and}~~ \kappa_{\sigma, \inf}^{\R}(E)=\kappa_{\sigma, \sup}^{\R}(E)=\kappa_{\vol, \inf}^{\R}(E)=\kappa_{\vol, \sup}^{\R}(E)=\frac{3}{2}.
$$
For any integer $k \geq 1$, let
$$
X:=\underbrace{Y \times \cdots \times Y}_{2k~\text{times}},
$$
and $p_i \colon X \to Y$ be the projection of $i$-th component $Y$ to itself for $1 \leq i \leq 2k$. Let
$$
D:=p_1^*E+ \cdots + p_{2k}^*E.
$$
The following can be checked easily:
$$
\begin{array}{l}
h^0(X, \lfloor mD \rfloor+p_1^*A + \cdots + p_{2k}^* A) = h^0(Y, \lfloor mE \rfloor+A)^{2k}~~\text{for any integer $m \geq 1$},\\
\vol_X(D + \varepsilon(p_1^*A + \cdots + p_{2k}^*A)) = \vol_Y(E+\varepsilon A)^{2k}~~\text{for any number $\varepsilon > 0$}.
\end{array}
$$
Then it follows that
$$
\kappa_{\star, \diamond}^{\R}(D)=\kappa_{\star, \diamond}(D)= 3k~~\text{for all $\star=\sigma, \vol$ and $\diamond=\inf, \sup$}.
$$
In particular, $\kappa_{\vol}(D) = \kappa_{\sigma}(D) = 3k$.
On the other hand, by Proposition \ref{prop:product},
$$
\nu_{\BDPP}(D) = 2k.
$$
Hence $\kappa_{\vol}(D)-\nu_{\BDPP}(D) =\kappa_{\sigma}(D)-\nu_{\BDPP}(D) = k$.
This proves Theorem \ref{thm:notcoincide}.
\end{example}

\begin{remark}\label{rem:lehmann(7)<=(1)}
The proof of $(7) \leq (1)$ in \cite[Proof of Theorem 6.2]{lehmann-nu} actually proves  the inequality $(1) \leq (7)$, that is, $\nu_{\BDPP}(D) \leq \kappa_{\vol}(D)$.
\end{remark}

\begin{remark}
In the setting of Proposition \ref{prop:product}, it is easy to see that
\begin{footnotesize}
$$
\kappa_{\star, \inf}^{\R}(p_1^*D_1 + p_2^*D_2) \geq \kappa_{\star, \inf}^{\R}(D_1) +\kappa_{\star, \inf}^{\R}(D_2)~~\text{and}~~\kappa_{\star, \sup}^{\R}(p_1^*D_1 + p_2^*D_2) \leq \kappa_{\star, \sup}^{\R}(D_1) +\kappa_{\star, \sup}^{\R}(D_2)
$$
\end{footnotesize}\\[-16pt]
where $\star = \sigma$ or $\vol$. It is unclear whether the inequalities are actually the equalities.
On the other hand, the inequalities can be strict without the $\R$-value variations; especially, the following additivity properties
$$
\kappa_{\vol}(p_1^*D_1 + p_2^*D_2) = \kappa_{\vol}(D_1) + \kappa_{\vol}(D_2)~~\text{and}~~ \kappa_{\sigma}(p_1^*D_1 + p_2^*D_2) = \kappa_{\sigma}(D_1) + \kappa_{\sigma}(D_2)
$$
do not hold.
Example \ref{ex:sigma,vol>bdpp} with $k=2$ gives a counterexample.
\end{remark}

%%%%%%%%%%%%%%%%%%%%%%%%%%%%%%%%%%%%%%%%%%%%%%%%%%%%%
\section{Abundant divisors}\label{sec:abdiv}
%%%%%%%%%%%%%%%%%%%%%%%%%%%%%%%%%%%%%%%%%%%%%%%%%%%%%

In this section, we prove Theorems \ref{thm:abundantdivisor} and \ref{thm:abconj<=>goodmin}.

\begin{proof}[Proof of Theorem \ref{thm:abundantdivisor}]
The `if' direction is clear, so we only have to show the `only if' direction.
By replacing $X$ by a sufficiently higher birational model of $X$, we may assume that there is a morphism $f \colon X \to Z$, which is the Iitaka fibration of $D$, such that  $Z$ is a smooth projective variety.
Let $n:=\dim X$, and $k:=\kappa(D)=\nu_{\BDPP}(D)=\dim Z \geq 1$.
If $D$ is big, then there is nothing to prove. Thus we can assume that $1 \leq k \leq n-1$.
By \cite[Corollary V.2.26]{nakayama} (see also \cite[Theorem 5.7]{lehmann-red}), it is sufficient to show that there is a birational morphism $\mu \colon W \to X$ from a smooth projective variety $W$ such that $\kappa_{\sigma}(P_{\sigma}(\mu^*D)|_{F'}) = 0$, where $F'$ is a general fiber of the surjective morphism $g:=f \circ \mu \colon W \to Z$.

\medskip

Let $F$ be a general fiber of $f$, and $C \subseteq F$ be a general smooth curve, which is a complete intersection of very ample divisors on $F$. We first show that $\vol_{X|C}^+(D)=0$.
Choose an admissible flag $Z_\bullet$ on $Z$ centered at a general point $f(F)$ of $Z$ and an admissible flag $F_\bullet$ on $F$ containing $C$.
Let $Y_\bullet$ be a fiber-type admissible flag on $X$ associated to $Z_\bullet$ and $F_\bullet$ (see \cite{CJPW}) so that
$$
Y_i = \begin{cases} f^{-1}(Z_i) & \text{if $0 \leq i \leq k$} \\ F_{i-k} & \text{if $k+1 \leq i \leq n$}. \end{cases}
$$
Since $\kappa_{\Delta}(D)=\kappa_{\BDPP}(D)=k$, we have
$$
\oklim_{Y_\bullet}(D) \subseteq \R_{\geq 0}^k \times \{ 0 \}^{n-k}.
$$
Take a sufficiently positive ample $\Z$-divisor $A$ on $X$ and an arbitrary number $\varepsilon > 0$.
By \cite[Theorem 1.1]{CPW}, we have
$$
\vol_{X|C}(D+\varepsilon A) = \vol_{\R}(\okbd_{Y_\bullet}(D+\varepsilon A)_{x_1=\cdots = x_{n-1}=0}).
$$
By taking $\varepsilon \to 0$, we see that
\begin{equation}\label{eq:vol_{X|C}^+(D)=0}
\vol_{X|C}^+(D) =  \vol_{\R}(\oklim_{Y_\bullet}(D)_{x_1=\cdots =x_{n-1}=0}) = 0.
\end{equation}

\medskip

Now, by \cite[Proposition 3.8]{lehmann-nu}, for each integer $m \geq 1$, there are a birational morphisms $\phi_m \colon X_m \to X$ centered at $\bm(D)$, an effective divisor $G$ on $X$, and a nef and big divisor $M_m$ on $X_m$ such that
$$
M_m \leq \frac{1}{m} P_{\sigma}(\phi_m^*(\lceil mD \rceil+ A)) \leq M_m + \frac{1}{m} \phi_m^*G.
$$
We may assume that $F_m:=(f \circ \phi_m)^{-1}(f(F))$ is a general fiber of $f \circ \phi_m \colon X_m \to Z$ and (by abuse of notation) $\phi_m \colon F_m \to F$ is a birational morphism. We denote by $C_m$ the strict transform of $C$ by $\phi_m$.
We have
$$
M_m \cdot C_m = \vol_{X_m|C_m}(M_m) \leq \vol_{X_m|C_m} \left( \phi_m^* \left( \frac{\lceil mD \rceil}{m}  + \frac{1}{m}A \right) \right) = \vol_{X|C}\left(\frac{\lceil mD \rceil}{m}  + \frac{1}{m}A \right).
$$
Since $\displaystyle \lim_{m \to \infty} \vol_{X|C}\left(\frac{\lceil mD \rceil}{m}  + \frac{1}{m}A \right) = \vol_{X|C}^+(D) = 0$ by (\ref{eq:vol_{X|C}^+(D)=0}), we get $\displaystyle \lim_{m \to \infty} M_m \cdot C_m = 0$.
As $\displaystyle \lim_{m \to \infty} \frac{1}{m} \phi_m^* G \cdot C_m =\lim_{m \to \infty} \frac{1}{m}G \cdot C = 0$, we have $\displaystyle \lim_{m \to \infty}P_{\sigma}\left(\phi^*_m D + \frac{1}{m}A \right) \cdot C_m = 0$,
and hence, we obtain
\begin{equation}\label{eq:P.C_m}
\lim_{m \to \infty}P_{\sigma}(\phi_m^*D) \cdot C_m = 0.
\end{equation}

\medskip

Let $D=P+N$ be the divisorial Zariski decomposition. Since $\dim \bm(P) \leq n-2$ and $F$ is a general fiber of $f$, we have
$\dim (\bm(P)\cap F) \leq \dim F - 2$. Then
$$
C \cap \bm(P) = C \cap (\bm(P) \cap F) = \emptyset
$$
because $C \subseteq F$ is general. Note that $P_{\sigma}(\phi_m^*D) = P_{\sigma}(\phi_m^*P)$ and $\phi_m^*(P|_F) = (\phi_m^*P)|_{F_m}$. Since $0 \leq P_{\sigma}((\phi_m^*P)|_{F_m}) - P_{\sigma}(\phi_m^*P)|_{F_m} \leq N_{\sigma}(\phi_m^*P)|_{F_m}$ and $\phi_m(N_{\sigma}(\phi_m^*P)) \subseteq \bm(P)$, it follows that
$$
P_{\sigma}((\phi_m^*P)|_{F_m}) \cdot C_m = P_{\sigma}(\phi_m^*P)|_{F_m}  \cdot C_m = P_{\sigma}(\phi_m^*P) \cdot C_m = P_{\sigma}(\phi_m^*D) \cdot C_m.
$$
Then by (\ref{eq:P.C_m}), we obtain
\begin{equation}\label{eq:P.C_m2}
\lim_{m \to \infty} P_{\sigma}((\phi_m^*P)|_{F_m}) \cdot C_m  = 0.
\end{equation}

\medskip

Suppose to the contrary that for any birational morphism  $\mu \colon W \to X$ from a smooth projective variety $W$, we have $\kappa_{\sigma}(P_{\sigma}(\mu^*D)|_{F'}) > 0$, where $F'$ is a general fiber of the surjective morphism $g:=f \circ \mu  \colon W  \to Z$. Recall that $P_{\sigma}(\phi_m^*D) = P_{\sigma}(\phi_m^*P)$. By Proposition \ref{prop:numdim=0,n-1,n}, we have $\kappa_{\vol,\zar}(P_{\sigma}(\phi_m^*P)|_{F_m})=\nu_{\BDPP}(P_{\sigma}(\phi_m^*P)|_{F_m}) > 0$. Thus
$$
\lim_{m \to \infty} P_{\sigma}((\phi_m^*P)|_{F_m}) \cdot C_m = \lim_{m \to \infty}\vol_{C_m}(P_{\sigma}((\phi_m^*P)|_{C_m})) >0,
$$
which is a contradiction to (\ref{eq:P.C_m2}). We complete the proof.
\end{proof}

Before giving the proof of Theorem \ref{thm:abconj<=>goodmin}, we recall the definitions of a klt pair and good minimal model. A log pair $(X, \Delta)$ is called \emph{klt} if $K_X+\Delta$ is $\Q$-Cartier and the discrepancy $a(E; X, \Delta) > -1$ for every prime divisor $E$ over $X$. We say that a $\Q$-factorial projective klt pair $(X, \Delta)$ has a \emph{good minimal model} $(X', \varphi_* \Delta)$ if there is a $(K_X+\Delta)$-negative birational map $\varphi \colon X \dashrightarrow X'$ to another $\Q$-factorial normal projective variety $X'$ such that$(X',\varphi_*\Delta)$ is klt and  $K_{X'}+\varphi_* \Delta$ is semiample. For more details on these definitions, see \cite{KM}.

\begin{proof}[Proof of Theorem \ref{thm:abconj<=>goodmin}]
Assume that $\kappa(K_X + \Delta) = \nu_{\BDPP}(K_X+\Delta) \geq 0$.
Let $g \colon W \to Z$ be a birational model of the Iitaka fibration of $K_X+\Delta$ such that $W, Z$ are smooth projective varieties and there is a birational morphism $\mu \colon W \to X$.
We write
$$
K_W + \Gamma = \mu^*(K_X + \Delta) + \sum a_i E_i,
$$
where $\Gamma$ is an effective divisor, each $E_i$ is a prime $\mu$-exceptional divisor, and $a_i > 0$ for every $i$. We may assume that $\Gamma$ has a simple normal crossing support and thus $(W, \Gamma)$ is a klt pair. Note that $P_{\sigma}(\mu^*(K_X + \Delta)) = P_{\sigma}(K_W + \Gamma)$. We claim that
\begin{equation}\label{eq:claimabconj<=>goodmin}
\nu_{\BDPP}(K_F + \Gamma|_F) = 0,
\end{equation}
where $F$ is a general fiber of $g$.
Granting the claim, we get $\kappa_{\sigma}(K_F + \Gamma|_F)=0$ by Proposition \ref{prop:numdim=0,n-1,n}. Then \cite[Corollary V.4.9]{nakayama} (see also \cite[Theorem 4.2]{GL}) implies that $(F, \Gamma|_F)$ has a good minimal model. By \cite[Theorem 4.4]{Lai}, $(W, \Gamma)$ has a good minimal model, and so does $(X, \Delta)$.

\medskip

We now prove the claim (\ref{eq:claimabconj<=>goodmin}).
There is an ample divisor $A_Z$ on $Z$ with $K_W + \Gamma \geq g^*A_Z$. Then
$$
\nu_{\BDPP}(K_X+\Delta) = \nu_{\BDPP}(K_W+\Gamma)=\nu_{\BDPP}(K_W + \Gamma+mg^*A_Z)~~\text{for any integer $m \geq 1$}.
$$
Let $D:=K_W + \Gamma+mg^*A_Z$ for some integer $m \geq 1$ such that $K_Z + mA_Z$ is ample.
Let $Z_\bullet$ be any admissible flag on $Z$, and $F_\bullet$ be an admissible flag on $F$ containing a positive volume subvariety of $K_F+\Gamma|_F$ such that $F_{\dim F}$ is a general point. Let $Y_\bullet$ be a fiber-type admissible flag on $X$ associated to $Z_\bullet$ and $F_\bullet$. Note that
$$
\okbd_{Z_\bullet}(mA_Z) \subseteq \oklim_{Y_\bullet}(D).
$$
In \cite[Proof of Theorem V.4.1]{nakayama}, Nakayama showed that there are ample divisors $A$ on $X$ and $H_Z$ on $Z$ such that $g_*\mathcal{O}_W(mD+A) \otimes \mathcal{O}_Z(H)$ is generically generated by global sections for any sufficiently divisible integer $m \geq 1$. This shows that
$$
\oklim_{F_\bullet}(K_F + \Gamma|_F) \subseteq \oklim_{Y_\bullet}(D)
$$
since $D|_F = K_F + \Gamma|_F$. Then we have
$$
\nu_{\BDPP}(D) \geq \dim \oklim_{Y_\bullet}(D) \geq \dim \okbd_{Z_\bullet}(mA_Z) + \dim \oklim_{F_\bullet}(K_F+\Gamma|_F).
$$
As $\nu_{\BDPP}(D) = \kappa(K_X+\Delta)=\dim Z=\dim \okbd_{Z_\bullet}(mA_Z) $, we obtain the claim  (\ref{eq:claimabconj<=>goodmin})
$$
\nu_{\BDPP}(K_F + \Gamma|_F)  = \dim \oklim_{F_\bullet}(K_F+\Gamma|_F) = 0.
$$

\medskip

Conversely, assume that $(X, \Delta)$ has a good minimal model $(X', \Delta')$. There is a morphism
$$
f \colon X' \to Z:=\Proj R(K_X+\Delta)
$$
such that $K_{X'}+\Delta' = f^*A$ for some ample $\Q$-divisor $A$ on $Z$. Then we have
$$
\nu_{\BDPP}(K_X+\Delta)=\nu_{\BDPP}(K_{X'}+\Delta')=\dim Z = \kappa(K_X+\Delta).
$$
We complete the proof.
\end{proof}

%%%%%%%%%%%%%%%%%%%%%%%%%%%%%%%%%%%%%%%%%%%%%%%%%%%%%


\begin{thebibliography}{ELMNP2}
%\bibliographystyle{amsplain}


\bibitem[B]{B} S. Boucksom, \textit{Divisorial Zariski decompositions on compact complex manifolds},
Ann. Sci. \'{E}c. Norm. Sup\'{e}r. (4) \textbf{37} (2004), 45--76.

\bibitem[BCHM]{BCHM} C. Birkar, P. Cascini, C. Hacon, and J. M\textsuperscript{c}Kernan, \textit{Existence of minimal models for varieties of log general type}, J. Amer. Math. Soc. \textbf{23} (2010), 405--468.

\bibitem[BDPP]{BDPP} S. Boucksom, J.-P. Demailly, M. P\u{a}un, and Th. Peternell, \textit{The pseudo-effective cone of a compact K\"{a}hler manifold and varieties of negative Kodaira dimension}, J. Algebraic Geom. \textbf{22} (2013), 201--248.

\bibitem[BFJ]{BFJ} S. Boucksom, C. Favre, and M. Jonhsson, \textit{Differentiability of volumes of divisors and a problem of Teissier}, J. Algebraic Geom. \textbf{18} (2009), 279--308.

\bibitem[CHPW]{CHPW} S. Choi, Y. Hyun, J. Park, and J. Won, \textit{Okounkov bodies associated to pseudoeffective divisors}, J. London Math. Soc. (2) \textbf{98} (2018), 170--195.

\bibitem[CJPW]{CJPW} S. Choi, S.-J. Jung, J. Park, and J. Won, \textit{A product formula for volumes of divisors via Okounkov bodies}, Int. Math. Res. Not. IMRN 2019, no. 22, 7118--7137.

\bibitem[CPW1]{CPW} S. Choi, J. Park, and J. Won, \textit{Okounkov bodies associated to pseudoeffective divisors II}, Taiwanese J. Math. \textbf{21} (2017), 602--620 (special issue for the proceedings of the conference Algebraic Geometry in East Asia 2016).

\bibitem[CPW2]{CPW2} S. Choi, J. Park, and J. Won, \textit{Local numerical equivalences and Okounkov bodies in higher dimensions}, to appear in Michigan Math. J.

\bibitem[DHP]{DHP} J.-P. Demailly, C. Hacon, and M. P\u{a}un, \textit{Extension theorems, non-vanishing and the existence of good minimal models}, Acta Math. \textbf{210} (2013), 203--259.

\bibitem[E]{E} Th. Eckl, \textit{Numerical analogues of the Kodaira dimension and the abundance conjecture}, Manuscripta Math. \textbf{150} (2016), 337--356.

\bibitem[ELMNP1]{elmnp1} L. Ein, R. Lazarsfeld, M. Musta\c{t}\u{a}, M. Nakamaye, and M. Popa, \textit{Asymptotic invariants of base loci}, Ann. Inst. Fourier \textbf{56} (2006), 1701--1734.

\bibitem[ELMNP2]{elmnp2} L. Ein, R. Lazarsfeld, M. Musta\c{t}\u{a}, M. Nakamaye, and M. Popa, \textit{Restricted volumes and base loci of linear series}, Amer. J. Math. \textbf{131} (2009), 607--651.

\bibitem[F]{F} O. Fujino, \textit{Corrigendum to ``On subadditivity of the logarithmic Kodaira dimension
"}, J. Math. Soc. Japan \textbf{72} (2020), 1181--1187.

\bibitem[GL]{GL} Y. Gongyo and B. Lehmann, \textit{Reduction maps and minimal model theory},
Compositio. Math. \textbf{149} (2013), 295--308

\bibitem[K]{K} Y. Kawamata, \textit{Pluricanonical systems on minimal algebraic varieties}, Invent. Math. 79 (1985), 567--588.

\bibitem[KK]{KK} K. Kaveh and A. G. Khovanskii, \textit{Newton convex bodies, semigroups of integral points, graded algebras and intersection theory}, Ann. of Math. (2) \textbf{176} (2012), 925--978.

\bibitem[KM]{KM} J. Koll\'{a}r and S. Mori, \textit{Birational geometry of algebraic varieties}, Cambridge Tracts in Math. \textbf{134} (1998), Cambridge Univ. Press, Cambridge.

\bibitem[Lai]{Lai} C.-J. Lai, \textit{Varieties fibered by good minimal models}, Math. Ann. \textbf{350}
(2011), 533--547.


\bibitem[LM]{lm-nobody} R. Lazarsfeld and M. Musta\c{t}\u{a}, \textit{Convex bodies associated to linear series}, Ann. Sci. \'{E}c. Norm. Sup\'{e}r. (4) \textbf{42} (2009), 783--835.

\bibitem[Leh1]{lehmann-nu} B. Lehmann, \textit{Comparing numerical dimensions}, Algebra Number Theory. \textbf{7} (2013), 1065--1100.
(Errata: \url{https://cpb-us-w2.wpmucdn.com/sites.bc.edu/dist/a/54/files/2020/01/numdimerrata.pdf})

\bibitem[Leh2]{lehmann-red} B. Lehmann, \textit{On Eckl's pesudo-effective reduction map}, Trans. Amer. Math. Soc. \textbf{366} (2014), 1525--1549. (Errata: \url{https://cpb-us-w2.wpmucdn.com/sites.bc.edu/dist/a/54/files/2020/01/pseferrata.pdf})

\bibitem[Les1]{john} J. Lesieutre, \textit{The diminished base locus is not always closed}, Compositio Math. \textbf{150} (2014), no. 10, 1729--1741.

\bibitem[Les2]{lesieutre} J. Lesieutre, \textit{Notions of numerical Iitaka dimension do not coincide}, to appear in J. Algebraic Geom.

\bibitem[M]{M} N. McCleerey, \textit{Volume of perturbations of pseudoeffective classes}, Pure Appl. Math. Q. \textbf{14} (2018), 607--616.

\bibitem[N]{nakayama} N. Nakayama, \textit{Zariski-decomposition and abundance},
MSJ Memoirs \textbf{14}. Mathematical Society of Japan, Tokyo, 2004.

\end{thebibliography}
\end{document}